\documentclass[12pt]{amsart}

\usepackage{amsmath}
\usepackage{amsfonts}
\usepackage{amssymb,enumerate}
\usepackage{amsthm}
\usepackage{fancyhdr}
\usepackage{tikz}
\usepackage{verbatim}
\usetikzlibrary{arrows,matrix}
\usepackage{hyperref}
\newcommand{\vect}[1]{\boldsymbol{#1}}

\newtheorem{lemma}{Lemma}[section]
\newtheorem{corollary}[lemma]{Corollary}
\newtheorem{proposition}[lemma]{Proposition}
\newtheorem{theorem}[lemma]{Theorem}
\newtheorem{question}[lemma]{Question}
\newtheorem{example}[lemma]{Example}
\newtheorem{definition}[lemma]{Definition}

\newtheorem{remark}[lemma]{Remark}

\title[Generators of a fraction of a numerical semigroup]{Generators of a fraction of a numerical semigroup}

\author{Alessio Moscariello}



\address[Alessio Moscariello]{Dipartimento di Matematica e Informatica, \ Universit\`a di Catania, \  Viale Andrea Doria 6, 
95125 Catania,Italy.}

\email{alessio.moscariello@studium.unict.it}

\address[Alessio Moscariello]{Scuola Superiore di Catania, \ Universit\`a di Catania, \  Via Valdisavoia 9, 
95123 Catania, Italy.}

\begin{document}

\maketitle
\begin{abstract}
Given a numerical semigroup $S$ and a positive integer $d$, the fraction  $\frac{S}{d}=\{ x \in \mathbb{N} \ | \ dx \in S\}$ is again a numerical semigroup. 
In this paper we determine a generating set for $\frac{S}{d}$ in terms of the minimal generators of $S$ and provide sharp upper bounds for the embedding dimension of $\frac{S}{d}$.

\end{abstract}
\section*{Introduction}
A numerical semigroup is an additive submonoid of $\mathbb{N}$ with finite complement in it. Numerical semigroups have been widely studied in the recent decades, because of their applications in algebraic geometry, number theory and coding theory.
In 2006 Rosales and Urbano-Blanco introduced the following construction: given a numerical semigroup $S$ and a positive integer $d$, the set 
$$
\frac{S}{d}=\{ x \in \mathbb{N} \ | \ dx \in S \}
$$
is again a numerical semigroup, called fraction or quotient of $S$ by $d$ (cf. \cite{RU}). The behavior of particular classes of semigroups with respect to this operation has been investigated, yielding surprising results: 
for instance, while every semigroup is a half of infinitely many symmetric numerical semigroup (cf. \cite{RG2}), only irreducible semigroups are halves of pseudo symemtric numerical semigroups (cf. \cite{R}); these results have been generalized by Swanson (cf. \cite{S}). Moreover, while Arf and saturated numerical semigroups are stable under taking fractions, semigroups of maximal embedding dimension are not (cf. \cite{DoS}). The construction appears in relation to problems in various areas of mathematics, including the study of graded free resolutions (cf. \cite{LLR}, \cite{MO}), Diophantine inequalities (cf. \cite{RR}, \cite{RU}),  $C^*$-algebras (cf. \cite{T}) and algebroid branches (cf. \cite{DS}).

Despite its ubiquity, the construction is not well understood. In particular, it is an open problem to relate the invariants of $\frac{S}{d}$ to those of $S$ (cf. \cite{DGR}). The aim of this paper is to study the generators of $\frac{S}{d}$. In Proposition \ref{gen} we determine a (non-minimal) generating set of $\frac{S}{d}$ in terms of the minimal generators of $S$. To this purpose we introduce a new combinatorial object, $d$-partitions, which plays the role of the partition of an integer in the context of modular arithmetic; this object may be of interest on its own (cf. Question \ref{prob}). The main consequences of our result are two sharp upper bounds for the embedding dimension of $\frac{S}{d}$, one depending on a partition of the minimal generating set of $S$ and the other depending only on the embedding dimension of $S$. Finally we apply our results to proportionally modular Diophantine inequalities and to give an alternative proof of the main result of \cite{RG2}.
\section{Preliminaries and $d$-partitions}
Let $\mathbb{N}$ be the set of non-negative integers. A \textbf{numerical semigroup} is a submonoid $S$ of $(\mathbb{N},+)$ such that $\mathbb{N} \setminus S$ is finite. Each numerical semigroup admits a unique minimal set of generators $G=\{g_1 < g_2 < \dots < g_k \}$, and the cardinality of this set of generators is called the \textbf{embedding dimension} and denoted with $\nu(S)$. Given a numerical semigroup $S$ and $d \in \mathbb{N} \setminus \{0\}$, the \textbf{fraction} $\frac{S}{d}$ is the numerical semigroup $$\frac{S}{d}:= \{x \in \mathbb{N} \ | \ dx \in S \}.$$
 
Now we introduce a new notion, whose scope is to provide a set of ``smallest" sequences of integers whose sum is a multiple of a given integer $d$.
\begin{definition}\label{dpart}
Let $d \in \mathbb{N} \setminus \{0\}$. A \textbf{$d$-partition} is a sequence of integers $\vect{\lambda} =(\lambda_1, \lambda_2,\dots,\lambda_m)$, where $\lambda_1 \le \dots \le \lambda_m$ satisfy the following condition:
\begin{enumerate}
\item $0 \le \lambda_i \le d-1$ for every $i=1,\dots,m$;
\item $\lambda_1+\dots+\lambda_m \equiv 0 \pmod{d}$;
\item There is no subsequence $\{\lambda_{i_1},\dots,\lambda_{i_k}\}$ of $\vect{\lambda}$ such that $\lambda_{i_1}+\dots+\lambda_{i_k} \equiv 0 \pmod{d}$.
\end{enumerate}
Denote the set of $d$-partitions with $\mathcal{P}(d)$.
\end{definition}
Obviously $(0) \in \mathcal{P}(d)$ for all $d \in \mathbb{N} \setminus \{0\}$. The partition $(0)$ is called the \textbf{trivial $d$-partition.}
Note also that if $\vect{\lambda}=(\lambda_1,\dots,\lambda_m) \in \mathcal{P}(d)$ and $\vect{\lambda} \neq (0)$ then $\lambda_i \neq 0$ for all $i=1,\dots,m$.
\begin{example}
If $d=1$ then the first condition forces $\mathcal{P}(1)=\{(0)\}$. For $d=2$ we obtain that, apart from the trivial $2$-partition, the only other sequence satisfying the conditions is $(1,1)$, therefore $\mathcal{P}(2)=\{(0),(1,1)\}$. For $d=3$ there are the three non-trivial sequences $(1,1,1),(1,2),(2,2,2)$, hence we have $\mathcal{P}(3)=\{(0),(1,1,1),(1,2),(2,2,2)\}$.
\end{example}
We immediately notice that if $(\lambda_1,\dots,\lambda_m)$ is a $d$-partition then $(d-\lambda_1,\dots,d-\lambda_m)$ is also a $d$-partition.
In the following we denote by $[m]_n$ (where $m,n \in \mathbb{N}, \ n > 0$) the euclidean remainder of the division of $m$ by $n$, i.e.   $$[m]_n=\min\{i \in \mathbb{N} \ \ |  \ i \equiv m \pmod n\}$$ 
We use this notation to see that the length of a $d$-partition is bounded above by $d$:
\begin{proposition}\label{pigeons}
Let $(\lambda_1,\dots,\lambda_m) \in \mathcal{P}(d)$. Then $m \le d$.
\end{proposition}
\begin{proof}
Consider the sequence $\vect{\Sigma}=(\Sigma_1,\dots,\Sigma_m)$ where $$\Sigma_k:=\left[\sum_{i=1}^{k} \lambda_i \right]_d.$$
Suppose that $m > d$. Since $0 \le \Sigma_k \le d-1$ for any $k \in \mathbb{N}$ then by the Pigeonhole Principle there exist two integers $1 \le i < j \le m$ such that $\Sigma_i=\Sigma_j$. Thus $\Sigma_j-\Sigma_i \equiv \lambda_{i+1}+\dots+\lambda_j \equiv 0 \pmod{d}$, contradicting the definition of $d$-partition.
\end{proof}
We next show that every sequence of integers satisfying the first two conditions of Definition \ref{dpart} can be decomposed in $d$-partitions.
\begin{proposition}
Let $d \in \mathbb{N} \setminus \{0\}$ and $(\lambda_1,\dots,\lambda_m)$ be a sequence of integers such that $0 \le \lambda_i < d$ and $$\lambda_1+\dots+\lambda_m \equiv 0 \pmod{d}.$$ Then this sequence is the union of elements of $\mathcal{P}(d)$.
\end{proposition}
\begin{proof}
If there is no subsequence of $\vect{\lambda}=(\lambda_1,\dots,\lambda_m)$ whose sum equals $0$ modulo $d$ then $(\lambda_1,\dots,\lambda_m) \in \mathcal{P}(d)$, thus the thesis holds.\\ If there exist such a subsequence $(\lambda_{i_1},\dots,\lambda_{i_k})$ then the sequence $\vect{\lambda}$ can be splitted into the two subsequences $(\lambda_{i_1},\dots,\lambda_{i_k})$ and $\vect{\lambda} \setminus (\lambda_{i_1},\dots,\lambda_{i_k})$. The two sequences obtained are shorter than the original one, therefore considering those sequences the thesis follow by infinite descent.
\end{proof}
\begin{corollary}\label{split}
Let $d \in \mathbb{N} \setminus \{0\}$ and let $(a_1,\dots,a_m)$ be a sequence of integers such that $$a_1+\dots+a_m \equiv 0 \pmod{d}.$$
Then the sequence $([a_1]_d,\dots,[a_m]_d)$ can be splitted into elements of $\mathcal{P}(d)$.
\end{corollary}
For a given set $A$, denote by $|A|$ its cardinality.
The last definition we give is the following:
\begin{definition}
Given a $d$-partition $\vect{\lambda}=(\lambda_1,\dots,\lambda_m)$ we define the \textbf{enumeration function} $\varphi_{\vect{\lambda}} : \{0,1,\dots,d-1\} \rightarrow \mathbb{N}$ as follows:
$$\varphi_{\vect{\lambda}}(n):=|\{\lambda_i \ | \ \lambda_i =n\}|$$
\end{definition}
We notice that $d$-partitions are in a sense the analogue of the well-studied partition of an integer (cf. \cite{A}); however this notion seems to be new in literature. While in this paper we are only interested in their applications to numerical semigroups, $d$-partitions might be an interesting tool in additive combinatorics. Therefore we conclude the section with the following motivating question.
\begin{question}\label{prob}
Let $d \in \mathbb{N} \setminus \{0\}$. Is it possible to characterize the set $\mathcal{P}(d)$ and give a formula for $|\mathcal{P}(d)|$?
\end{question}
\section{Main results}
Let $S$ be a numerical semigroup, $d \in \mathbb{N} \setminus \{0\}$ and let $G=\{g_1,\dots,g_{\nu(S)}\}$ be the set of its minimal generators. For $i=0,\dots,d-1$ denote by $G_i$ the set $$G_i := \{g \in G \ | \ g \equiv i \pmod{d} \}.$$
Notice that if $(\lambda_1,\dots,\lambda_m) \in \mathcal{P}(d)$ and $g_{\gamma_1},\dots,g_{\gamma_m}$ are such that $g_{\gamma_i} \in G_{\lambda_i}$ for all $i=1,\dots,m$, then $g_{\gamma_1}+\dots,g_{\gamma_m} \equiv \lambda_1+\dots+\lambda_m \equiv 0 \pmod{d}$, and therefore $\frac{g_{\gamma_1}+\dots,g_{\gamma_m}}{d} \in \mathbb{N}$.
Actually we are going to prove that all the generators of $\frac{S}{d}$ are of this form, by constructing a generating set for $\frac{S}{d}$:
\begin{proposition}\label{gen}
Let $S=\langle G \rangle$ be a numerical semigroup, and let $\vect{\lambda}=(\lambda_1,\dots,\lambda_m) \in \mathcal{P}(d)$. Let
$$\Gamma_{\vect{\lambda}}\left(\frac{S}{d}\right):=\left\{ \frac{g_{\gamma_1}+\dots+g_{\gamma_m}}{d} \ | \ g_{\gamma_i} \in G_{\lambda_i} \ \text{for all} \ i=1,\dots,m \right\}.$$
Then the set $$\Gamma\left(\frac{S}{d} \right):= \bigcup_{\vect{\lambda} \in \mathcal{P}(d)}  \Gamma_{\vect{\lambda}} \left(\frac{S}{d} \right)$$ is a generating system of $\frac{S}{d}$.
\end{proposition}
\begin{proof}
Consider $x \in \frac{S}{d}$. Then $dx \in S$, or equivalently  $dx=a_1g_1+\dots+a_{\nu(S)}g_{\nu(S)}$, hence $x=\frac{a_1g_1+\dots+a_{\nu(S)}g_{\nu(S)}}{d}.$ Since $x \in \mathbb{N}$ we have $$\sum_{j=1}^{\nu(S)} a_jg_j = \sum_{j=1}^{\nu(S)} \underbrace{(g_j+\dots+g_j)}_{a_j} \equiv 0 \pmod{d}$$ and then by Corollary \ref{split} the extended sequence $(\underbrace{[g_j]_d,\dots,[g_j]_d}_{a_j},\dots)$ is the union of some elements of $\mathcal{P}(d)$. Consider the sequences of generators associated to these elements of $\mathcal{P}(d)$. The sum of the elements of each of those sequences, divided by $d$, is an element of $\Gamma\left(\frac{S}{d}\right)$. Hence $x$ can be expressed as a linear combination of elements of $\Gamma\left(\frac{S}{d}\right)$, and the proof is concluded.
\end{proof}
\begin{remark}\label{notmin}
The statement of Proposition \ref{gen} actually holds even if $G$ is not the minimal system of generators of $S$. The only thing used in the proof is that $G$ is a generating system of $S$, regardless of its minimality: it's trivial to see that if we consider a non-minimal generating system $G$ we will obtain a larger generating set $\Gamma\left(\frac{S}{d}\right)$. However, to keep a lighter notation, we do not specify the generating system in $\Gamma\left(\frac{S}{d}\right)$ as in the following parts we will mostly use the minimal generating system for $S$. 
\end{remark} 
\begin{corollary}\label{genS}
Let $S$ be a numerical semigroup, let $\lambda,d \in \mathbb{N}$ be such that $\lambda < d$. Then $\frac{G_\lambda}{\gcd(\lambda,d)} \subseteq \Gamma\left(\frac{S}{d}\right)$. In particular, if $\gcd(\lambda,d)=1$, then $G_\lambda \subseteq \Gamma\left(\frac{S}{d}\right)$.
\end{corollary}
\begin{proof}
Take $g \in G_\lambda$. Obviously $(\overbrace{\lambda, \dots, \lambda}^{\frac{d}{\gcd(\lambda,d)}}) \in \mathcal{P}(d)$, therefore $$\frac{g}{\gcd(\lambda,d)}=\frac{\overbrace{g+\dots+g}^{\frac{d}{\gcd(\lambda,d)}}}{d} \in \Gamma\left(\frac{S}{d}\right)$$ and the claim follows.
\end{proof}
\begin{remark}
The set of generators $\Gamma\left(\frac{S}{d}\right)$ is not minimal in general. In fact, consider the numerical semigroup $S=\langle 7,9,13 \rangle$ and $d=3$. Since $G_1=\{7,13\}$ and $G_0=\{9\}$, Corollary \ref{genS} states that $\frac{G_0}{3}=\{3\}$ and $G_1=\{7,13\}$ are subsets of $\Gamma\left(\frac{S}{d}\right)$. However $13=7+3\cdot 2$, therefore $13$ is a linear combination of other elements of $\Gamma\left(\frac{S}{d}\right)$, thus it is not a minimal generator of $\frac{S}{d}$. 
\end{remark} 
Since $\Gamma\left(\frac{S}{d}\right)$ contains the minimal generators of $\frac{S}{d}$, we can use Proposition \ref{gen} to give upper bounds for $\nu\left(\frac{S}{d}\right)$. In the next Lemma, we use the convention that $\binom{n}{0}=1$ for all integers $n$ and $\binom{n}{k}=0$ if $n < k$ and $k > 0$.
\begin{lemma}\label{gammal}
Let $\vect{\lambda} \in \mathcal{P}(d)$. Then 
$$\left|\Gamma_{\vect{\lambda}}\left(\frac{S}{d}\right)\right| \le \prod_{i=0}^{d-1} \binom{|G_i|+\varphi_{\vect{\lambda}}(i)-1}{\varphi_{\vect{\lambda}}(i)}.$$
\end{lemma}
\begin{proof}
From the definition of $\Gamma_{\vect{\lambda}}\left(\frac{S}{d} \right)$ it follows that every element of this set is associated to a combination of elements of the sets $G_i$, taking $\varphi_{\vect{\lambda}}(i)$ from each of them. The number of such combinations is the product of the number of combinations (with repetitions) of $\varphi_{\vect{\lambda}}(i)$ objects of $G_i$. The inequality is then a direct consequence of the formula for the number of combinations with repetitions.
\end{proof}
\begin{theorem}\label{gamma}
Let $S$ be a numerical semigroup, and $d \in \mathbb{N} \setminus \{0\}$. Then $$\nu\left(\frac{S}{d}\right) \le \sum_{\vect{\lambda} \in \mathcal{P}(d)} \prod_{i=0}^{d-1} \binom{|G_i|+\varphi_{\vect{\lambda}}(i)-1}{\varphi_{\vect{\lambda}}(i)}$$ and this bound is sharp.
\end{theorem} 
\begin{proof}
By Proposition \ref{gen}, the definition of $\Gamma\left(\frac{S}{d}\right)$ and Lemma \ref{gammal} we obtain 
$$\nu\left(\frac{S}{d}\right) \le \left|\Gamma\left(\frac{S}{d}\right)\right| \le \sum_{\vect{\lambda} \in \mathcal{P}(d)} \left|\Gamma_{\vect{\lambda}}\left(\frac{S}{d}\right) \right|\le \sum_{\vect{\lambda} \in \mathcal{P}(d)} \prod_{i=0}^{d-1} \binom{|G_i|+\varphi_{\vect{\lambda}}(i)-1}{\varphi_{\vect{\lambda}}(i)}$$ that proves the bound.

Consider now, for $d \in \mathbb{N} \setminus \{0\}$, the numerical semigroup $S_d=\langle d+1, d^2, d^2+2d,d^2+3d,\dots,d^2+(d-1)d \rangle$. In this case we have $|G_0|=d-1$, $|G_1|=1$ and $|G_i|=0$ otherwise. Then all the binomials associated to any $d$-partitions apart $(0)$ and $(1,1,\dots,1)$ are equal to zero, thus obtaining
$$\nu\left(\frac{S}{d}\right) \le \binom{|G_0|+1-1}{1}+\binom{|G_1|+d-1}{d}=d-1+1=d$$ But $\frac{S}{d}=\{0,d,d+1,\dots\}$, hence $\nu\left(\frac{S}{d}\right)=d$, and the bound is sharp.
\end{proof}
The bound given in Theorem \ref{gamma} is quite involved, as it depends on the partition $G_i$ and the functions $\varphi_{\vect{\lambda}}$.
However we can derive a sharp bound for $\nu\left(\frac{S}{d}\right)$ only in terms of the embedding dimension of $S$ and $d$.
\begin{theorem}\label{vs}
Let $S$ be a numerical semigroup, and let $d \in \mathbb{N} \setminus \{0\}$. Then
$$\nu\left(\frac{S}{d} \right) \le \binom{\nu(S)+d-1}{d}$$ and this bound is sharp.
\end{theorem}
\begin{proof}
Consider the set $$\mathcal{C}=\{(a_1,\dots,a_d) \ | \ a_i \in G\}.$$ The elements of $\mathcal{C}$ are the sequences of elements of $G$ of lenght $d$. It is immediate that $|\mathcal{C}|=\binom{\nu(S)+d-1}{d}$. 
Define the set $$\mathcal{X}:=\{(g_{\gamma_1},\dots,g_{\gamma_m}) \ |  \ g_{\gamma_i} \in G_{\lambda_i} \ \text{for any} \ i=1,\dots,m, \ \ (\lambda_1,\dots,\lambda_m) \in \mathcal{P}(d) \}$$
Each element of $\Gamma\left(\frac{S}{d}\right)$ is associated to at least one element of $\mathcal{X}$, thus $\left|\Gamma\left(\frac{S}{d}\right)\right| \le |\mathcal{X}|$. Consider now the function $\sigma: \mathcal{X} \rightarrow \mathcal{C}$ defined by $$\sigma((g_{\gamma_1},\dots,g_{\gamma_m}))=(g_{\gamma_1},\dots,g_{\gamma_m},g_1,\dots,g_1).$$ Note that this function is well defined by Proposition \ref{pigeons}. Moreover, if there exist $(g_{\gamma_1},\dots,g_{\gamma_m})$ and $(g_{\gamma_1},\dots,g_{\gamma_k})$ such that $\sigma((g_{\gamma_1},\dots,g_{\gamma_m}))=\sigma((g_{\gamma_1},\dots,g_{\gamma_k}))$ then we must have that one of the two sequences is contained in the other one, thus leading to a contradiction. Then $\sigma$ is injective, and $|\mathcal{X}| \le |\mathcal{C}|$. Finally, from Proposition \ref{gen} we obtain $$\nu\left(\frac{S}{d}\right) \le \left|\Gamma\left(\frac{S}{d}\right)\right| \le |\mathcal{X}| \le |\mathcal{C}| = \binom{\nu(S)+d-1}{d}$$ and the bound is proven.

For the sharpness of this bound, take $n \ge 2$, $d \ge 2$ and consider the following family of sets $$G_{n,d}:=\{g_i:=d^n-d-1+2^{i-1}d \ | \ i=1,\dots,n\}.$$
Notice that $g_1=d^n-1$, $g_2=d^n+d-1$ and $\gcd(g_1,g_2)=1$, thus we can build 
the family of numerical semigroups $S_{n,d}:=\langle G_{n,d} \rangle$. Since $$\max G_{n,d}=d^n-d-1+2^{n-1}d \le 2d^n-d-1 < 2\cdot d^n-2 = 2 \cdot \min G_{n,d}$$  $G_{n,d}$ is the minimal generating system of $S_{n,d}$ and $\nu(S_{n,d})=|G_{n,d}|=n$. Since all elements of $G_{n,d}$ are equal to $d-1$ modulo $d$ then $(G_{n,d})_{d-1}=G_{n,d}$. Thus, when constructing $\Gamma\left(\frac{S_{n,d}}{d}\right)$ we only have to consider the $d$-partition $(d-1,d-1,\dots,d-1)$, then
$$\Gamma\left(\frac{S_{n,d}}{d}\right)=\left\{ \frac{g_{\gamma_1}+\dots+g_{\gamma_d}}{d} \ | \ g_{\gamma_i} \in G_{n,d} \ \text{for any} \ i=1,\dots,d \right\}.$$
But since $$\frac{g_{\gamma_1}+\dots+g_{\gamma_d}}{d}=d^n-d-1+2^{\gamma_1-1}+2^{\gamma_2-1}+\dots+2^{\gamma_d-1}$$ we deduce that different combinations of generators give different elements in $\Gamma\left(\frac{S_{n,d}}{d}\right)$, thus $$\left|\Gamma\left(\frac{S_{n,d}}{d}\right)\right|=\binom{n-d+1}{d}=\binom{\nu(S)-d+1}{d}.$$ We only need to show that $\Gamma\left(\frac{S_{n,d}}{d}\right)$ is the minimal system of generators of $\frac{S_{n,d}}{d}$: noticing that $$\min \Gamma\left(\frac{S_{n,d}}{d}\right)=\frac{g_1+\dots+g_1}{d}=g_1  \ \ \text{and} \ \ \max \Gamma\left(\frac{S_{n,d}}{d}\right)=\frac{g_n+\dots+g_n}{d}=g_n$$ then $\max G_{n,d}=g_n< 2g_1 < 2 \cdot \min G_{n,d}$, hence all the generators of $\Gamma\left(\frac{S_{n,d}}{d}\right)$ are minimal.
\end{proof}
Remark that the semigroups $S_{n,d}$ used in Theorem \ref{vs} actually satisfy both equalities of Theorem \ref{gamma} and \ref{vs}.
\section{Applications}
We conclude the paper by showing that Proposition \ref{gen} can be used to give alternative proofs of some known in literature. The first application regards proportionally modular semigroups (cf. \cite{RU}).
Notice that Proposition \ref{gen} is particularly helpful if the semigroup we consider has ``small" embedding dimension. In fact, the $d$-partitions that are actually involved with elements of $\Gamma\left(\frac{S}{d}\right)$ are formed by the $0 \le \lambda_i \le d-1$ such that $G_{\lambda_i} \neq \emptyset$. By applying Proposition \ref{gen} we can prove the three Corollaries 18,19 and 20 from \cite{RU}. We show a proof of Corollary 19: the other two are very similar.
\begin{corollary}[\protect{\cite[Corollary 19]{RU}}]
Let $n_1$ and $n_2$ be positive integers such that $n_1,n_2$ and $3$ are pairwise relatively prime, and let $S=\langle n_1,n_2 \rangle$. Then:
\begin{enumerate}
\item If $n_1+n_2 \equiv 0 \pmod{3}$ then $\frac{S}{3}=\left \langle n_1,n_2,\frac{n_1+n_2}{3} \right \rangle$.
\item If $n_1+2n_2 \equiv 0 \pmod{3}$ then $\frac{S}{3}=\left \langle n_1,n_2,\frac{n_1+2n_2}{3},\frac{2n_1+n_2}{3} \right \rangle$.
\end{enumerate}
\end{corollary}
\begin{proof}
First of all we know that $\mathcal{P}(3)=\{(0),(1,1,1),(1,2),(2,2,2)\}$. We construct $\Gamma\left(\frac{S}{3}\right)$ in both cases, and the thesis will follow from Proposition \ref{gen}:
\begin{enumerate}
\item If $n_1+n_2 \equiv 0 \pmod{3}$ then we can suppose without loss of generality that $n_1 \equiv 2 \pmod{3}$ and $n_2 \equiv 1 \pmod{3}$. Therefore $$\Gamma\left(\frac{S}{3}\right)=\left\{\frac{n_1+n_1+n_1}{3},\frac{n_2+n_2+n_2}{3},\frac{n_1+n_2}{3}\right\}=\left\{n_1,n_2,\frac{n_1+n_2}{3}
\right\}.$$
\item If $n_1+2n_2 \equiv 0 \pmod{3}$ then $n_1 \equiv n_2 \pmod{3}$. Suppose that $n_1 \equiv 1 \pmod{3}$. Hence $G_1=\{n_1,n_2\}$, and the only $3$-partition we have to considerate is $(1,1,1)$. Then $$\Gamma\left(\frac{S}{3}\right)=\left\{\frac{n_1+n_1+n_1}{3},\frac{n_2+n_2+n_2}{3},\frac{n_1+n_1+n_2}{3},\frac{n_1+n_2+n_2}{3}\right\}$$ that is our claim. The case $n_1 \equiv 2 \pmod{3}$ is identical. 
\end{enumerate}
\end{proof}
The second application is related to symmetric numerical semigroups.
Since a numerical semigroup has finite complement in $\mathbb{N}$ we define the \textbf{Frobenius number} of $S$ as the greatest element in $\mathbb{N} \setminus S$, denoted as $F(S)$. We say that a numerical semigroup is \textbf{symmetric} if for every $z \in \mathbb{Z}$ we have either $z \in S$ or $F(S)-z \in S$. \begin{proposition}[\cite{R2}]\label{T}
Let $S$ be a numerical semigroup such that $F(S)$ is odd.
Then the set $$T := S \cup \left\{ x \in \mathbb{N} \setminus \{0\} | x \ge \frac{F(S)}{2}, \ F(S)-x \not \in S \right\}$$ is a symmetric numerical semigroup such that $F(S)=F(T)$.
\end{proposition}  
We will use this result to give an alternative proof of the following:
\begin{corollary}[\protect{\cite[Theorem 5]{S}}]\label{sym}
Let $S$ be a numerical semigroup with minimal system of generators $\{g_1,\dots,g_n\}$ and let $d \in \mathbb{N} \setminus \{0\}$ be such that $d \ge 2$. Then there exist infinitely many symmetric numerical semigroups $T$ such that $S=\frac{T}{d}$.
\end{corollary}
\begin{proof}

Take $\rho \in \mathbb{N} \setminus \{0\}$ such that $\rho$ is an odd integer not multiple of $d$ and $\rho > 2dF(S)$. The numerical semigroup $$S_{\rho} := \langle dg_1,\dots,dg_n \rangle \cup \{\rho+1,\rho+2, \dots \}$$ is such that $F(S_{\rho})=\rho$ is odd. Therefore by Proposition \ref{T} the set $$T_{\rho}:= S_{\rho} \cup \left\{x \in \mathbb{N} \setminus \{0\} \ | \ x \ge \frac{F(S_{\rho})}{2}, \  F(S_{\rho})-x \not \in S_{\rho} \right\}$$ is a symmetric numerical semigroup such that $F(T_\rho)=F(S_{\rho})=\rho$.
A generating set of $T_{\rho}$ is $$\mathcal{G}=\{dg_1,\dots,dg_n, \rho_i+1, \rho_i+2,\dots, 2\rho+1,x_1,\dots, x_k\}$$ where $x_1,\dots,x_k$ are such that $x_j \ge \frac{\rho}{2}$ and $\rho - x_j \not \in S_{\rho}$ for any $j=1,\dots, k$. Since $x_j \ge \frac{\rho}{2} > dF(S)$ for any $j=1,\dots,k$, all elements of $\mathcal{G}$ aside from $g_1,\dots,g_n$ are greater than $dF(S)$.
Applying Proposition \ref{gen} and Remark \ref{notmin} to this generating set we obtain that $$\Gamma\left(\frac{T_{\rho}}{d}\right) =\left\{ \frac{\omega_{\gamma_1}+\dots+\omega_{\gamma_m}}{d} \ | \ \omega_{\gamma_i} \in G_{\lambda_i} \ \text{for all} \ (\lambda_1,\dots,\lambda_m) \in \mathcal{P}(d), \ i=1,\dots,m \right\}$$ is a generating system of $\frac{T_{\rho}}{d}$.  Naturally $\{dg_1,\dots,dg_n\} \subseteq \mathcal{G}_0$, hence $\{g_1,\dots,g_n\} \subseteq \Gamma\left(\frac{T_\rho}{d}\right)$. 

If $\omega \in \Gamma\left(\frac{T_{\rho}}{d}\right) \setminus \{g_1,\dots,g_n\}$, since $\{dg_1,\dots,dg_n\} \subseteq \mathcal{G}_0$ we must have that $\omega$ is of the form $\omega=\frac{\omega_{\gamma_1}+\dots+\omega_{\gamma_m}}{d}$ where at least one (actually all) of the $\omega_{\gamma_i}$ is in $\mathcal{G} \setminus \{dg_1,\dots,dg_n\}$: this directly implies $\omega_{\gamma_i} > dF(S)$ and $\omega > F(S)$. Then all elements of $\Gamma\left(\frac{T_{\rho}}{d}\right) \setminus \{g_1,\dots,g_n\}$ are greater than $F(S)$, thus are linear combination of $g_1,\dots,g_n$.  But $\{g_1,\dots,g_n\}$ is the minimal set of generators of $S$, hence $\frac{T_{\rho}}{d}=S$. Since there are infinitely many choices of $\rho$, we are done.
\end{proof}
Moreover, considering even values of $\rho$ (taking apart the case $d=2$) and with slight modifications of this proof and Proposition \ref{T} we can prove the following:
\begin{corollary}[\protect{\cite[Theorem 6]{S}}]
Let $S$ be a numerical semigroup, and let $d \in \mathbb{N} \setminus \{0\} , d \ge 3$. Then there exist infinitely many pseudo-symmetric numerical semigroups $T$ such that $S=\frac{T}{d}$.
\end{corollary}
The main idea behind the proof of Corollary \ref{sym} is to construct a convenient symmetric numerical semigroup $T$ depending on $\rho$ such that all elements of $\Gamma\left(\frac{T}{d}\right) \setminus \{g_1,\dots,g_n\}$ are forced to be in $S$, therefore implying $\frac{T}{d}=S$. This idea allows us to use the direct expression of $T$ without resorting to other definitions (the original proof mainly used properties of $d$-symmetric numerical semigroups), just by looking at the shape of the generators of $\frac{T}{d}$ when $\rho$ is larger than a certain bound. Furthermore, it's interesting to see that the numerical semigroup $T$ is not related to the pseudo-Frobenius numbers of $S$, which are used in the original proof to state that the numerical semigroup presented is symmetric.
\section*{Acknowledgements}
I would like to thank Alessio Sammartano for his helpful comments and suggestions on this work.

\end{document}